\documentclass[12pt]{article}

\usepackage[font=footnotesize,margin=1cm]{caption}

\usepackage[margin=1in]{geometry}

\usepackage{setspace}
\usepackage{amsmath,amsthm,amssymb}
\usepackage{mathtools}

\usepackage{graphicx,float,stmaryrd}
\usepackage{subcaption}
\usepackage[export]{adjustbox}

\usepackage{epsfig,epsf,rotating}

\usepackage{algpseudocode,algorithm,algorithmicx}

\usepackage{shortvrb,psfrag}

\newtheorem{theorem}{Theorem}
\newtheorem{lemma}{Lemma}
\usepackage{color}

\allowdisplaybreaks

\floatstyle{plain}
\newfloat{myalgo}{tbhp}{mya}
\newenvironment{Algorithm}[1][tbh]%
{\begin{myalgo}[#1]
\centering
\begin{minipage}{16.5cm}
\begin{algorithm}[H]}
{\end{algorithm}
\end{minipage}
\end{myalgo}}

\DeclareMathOperator*{\argmin}{arg\,min}
\DeclareMathOperator*{\argmax}{arg\,max}
\DeclareMathOperator*{\sign}{sign}

\makeatletter
\newcommand*{\rom}[1]{\expandafter\@slowromancap\romannumeral #1@}
\makeatother

\makeatletter
\newcommand{\eqnum}{\refstepcounter{equation}\textup{\tagform@{\theequation}}}
\makeatother

\makeatletter
\newcommand{\fixed@sra}{$\vrule height 2\fontdimen22\textfont2 width 0pt\shortrightarrow$}
\newcommand{\shortarrow}[1]{%
  \mathrel{\text{\rotatebox[origin=c]{\numexpr#1*45}{\fixed@sra}}}
}
\makeatother

\newcommand{\change}{}

\title{A Convex Program for Mixed Linear Regression with a Recovery Guarantee for Well-Separated Data}
\author{Paul Hand\thanks{Department of Computational and Applied Mathematics, Rice University, TX.}\ \ and Babhru Joshi\footnotemark[1]}
\date{\today}

\begin{document}
\maketitle 
 
\begin{abstract}
	{We introduce a convex approach for mixed linear regression over $d$ features. This approach is a second-order cone program, based on L1 minimization, which assigns an estimate regression coefficient in $\mathbb{R}^{d}$ for each data point. These estimates can then be clustered using, for example, $k$-means. For problems with two or more mixture classes, we prove that the convex program exactly recovers all of the mixture components in the noiseless setting under technical conditions that include a well-separation assumption on the data. Under these assumptions, recovery is possible if each class has at least $d$ independent measurements. We also explore an iteratively reweighted least squares implementation of this method on real and synthetic data.}
	{mixed linear regression; L1 minimization; second-order cone programming; clustering; iteratively reweighted least squares.}
\end{abstract}


\section{Introduction}
\label{ch:Intro}
We study the mixed linear regression problem with $k$ classes and $m$ data points. Let $\{S_{p}\}_{p=1}^{k}$ be a partition of ${\change \{1,\dots,m\}}$, where $i \in S_{p}$ means that the $i$th data point belongs to the $p$th class. In class $p$, we assume that the data is generated by a linear process with regression coefficients $\beta_{p} \in \mathbb{R}^{d}$. Let the measurements $\{(a_{i},b_{i})\}_{i=1}^{m} \subset \mathbb{R}^{d}\times \mathbb{R}$ be such that for all $p \in {\change  \{1,\dots,k\}}$, 
\begin{equation}\label{problem}
a_{i}^{\top}\beta_{p} = b_{i} \text{ if } i \in S_{p}.	
\end{equation}
 The mixed linear regression problem is to simultaneously estimate $\{\beta_{p}\}_{p =1}^{k}$ and $\{S_{p}\}_{p=1}^{k}$ from  $\{(a_{i},b_{i})\}_{i=1}^{m}$. In the above problem formulation, let $n_{p} = |S_{p}|$. We define $\ell_{i}$ as the $p$ such that $i\in S_{p}$, i.e. $\ell_{i}$ is the label of the $i$th data point.   

Mixed linear regression has broad applications in scenarios that require disentangling data into multiple latent classes and estimating the parameters. One application is resource management in health care, where the classes are patients (with and without illness) and the mixed data is usage of resource and medical care \cite{Deb}. Other applications include music tone perception \cite{Cohen}, subspace clustering \cite{Vidal} and population classification \cite{Leisch}.

Recent work on mixed linear regression has mostly focused on mixtures of two classes, i.e. $k = 2$ \cite{Yi1, Chen1}. For independent, Gaussian measurements, the authors in \cite{Yi1} develop an initialization procedure for the Expectation Minimization (EM) algorithm. This initialization is based on a grid search on the eigenspace of a matrix of second order moments of measurements. In the noiseless setting, the authors prove that their initialization, followed by an EM algorithm recovers the two regression coefficients with high probability if there are $O(d\log^{2}d)$ Gaussian measurements. In \cite{Chen1}, the authors lift the mixed linear regression problem to a low rank non-symmetric matrix completion problem. In the noiseless case with $O(d)$ sub-gaussian measurements, the program exactly recovers the two mixture components with high probability. In both of these approaches, the number of mixture components were restricted to exactly two.

For mixed linear regression with two or more mixture components, tensor-based methods have been introduced in \cite{Chag} and, very recently, in \cite{Yi2}. In \cite{Chag}, low-rank tensor regression and tensor factorization is used to recover the mixture components. In the noiseless setting{\change, i.e. for all $p \in \{1,\dots,k\}$, $a_{i}^{\top}\beta_{p} = b_{i}$ if $i \in S_{p}$}, their recovery theorem does not guarantee exact recovery for any $m$. Instead, it establishes that the recovery error is $O(\frac{1}{\sqrt{m}})$. In \cite{Yi2}, the method of moments is used to generate an initialization for the EM algorithm. For i.i.d. Gaussian measurements, the algorithm provably recovers the mixture components if the number of measurements is $O(k^{10}d)$. {\change For identification of $k$ classes, there must be at least $kd$ measurements. As a result,} this approach has an optimal sample complexity in the number of feature elements, $d$, but its scaling with respect to the number of classes, $k$, is suboptimal. Another drawback of these tensor-based methods is that they are computationally expensive because they operate in high-dimensional spaces.

The work in the present paper brings ideas from convex clustering to the problem of mixed linear regression.  The algorithm we introduce is inspired by the algorithm in \cite{Hocking}.  In that algorithm, a $\ell_1$ minimization in the form of a fused lasso \cite{Tib2} acts to cluster noisy points in $\mathbb{R}^{d}$.  The problem of mixed linear regressions can be viewed as a challenging generalization of clustering in $\mathbb{R}^{d}$ where every measurement is subject to a {\change codimension-1} ambiguity.  We note in particular that in the noiseless case, $\mathbb{R}^{d}$ clustering is trivial whereas mixed linear regression is not.  Consequently, this paper focuses only on the case of noiseless measurements in order to grapple with the difficulty of this codimension-1 ambiguity.

The contribution of the present paper is to introduce a convex algorithm for the mixed linear regression problem with two or more components. In the noiseless setting and under some technical conditions that include well-separation of data, this algorithm exactly recovers the mixture components with $kd$ independent measurements. The convex algorithm is based on an $\ell_{1}$ minimization that assigns an estimate regression coefficient in $\mathbb{R}^{d}$ for each data point. These estimates can then be clustered by, for example, $k$-means. An estimate for each class can then be found by running $k$ separate regressions.


\subsection{Convex Program and Main Result}
\label{sec:Fused Lasso and Main Results}
We introduce a two-step approach for mixed linear regression. This approach introduces a free variable for each data point that acts as the estimate of the mixture component corresponding to that data point. The first step is to solve a second-order cone program to obtain estimated mixture components for each data point, and the second step is to cluster these estimates using $k$-means. The details of these steps are:
\begin{itemize}
\item[1.] Solve the following second order cone program:	
\begin{equation} \label{lasso}
\begin{aligned}
&\argmin_{z_{1}, \dots, z_{m} \in \mathbb{R}^{d}} \sum_{i=1}^{m}\sum_{j = 1}^{m} \|z_{i} - z_{j}\|_{2}\\
&\text{subject to }  a_{i}^{\top}z_{i} = b_{i}, \ i\in {\change \{1,\dots,m\}}.
\end{aligned}
\end{equation}
Each $z_{i}$ in \eqref{lasso} is constrained to belong to the hyperplane corresponding to the measurement $(a_{i},b_{i})$, as per \eqref{problem}. If $\{z_{i}^{\natural}\}_{i=1}^{m}$ minimizes \eqref{lasso}, then $z_{i}^{\natural}$ is an estimate of $\beta_{p}$ for the class $p$ such that $i \in S_{p}$. 
\item[2.] Cluster $\{z_{i}^{\natural}\}_{i=1}^{m}$ using $k$-means and estimate the mixture components corresponding to each class by running $k$ separate regressions. 
\end{itemize}

 Program \eqref{lasso} is an $\ell_{1}$ minimization over all pairs $(i,j)$. Due to the sparsity promoting property of $\ell_{1}$ minimization, in some cases, \eqref{lasso} can find a solution where $z_{i}^{\natural} = z_{j}^{\natural}$ for many pairs $(i,j)$. Let $\{z_{i}^{\sharp}\}_{i=1}^{m}$ be the candidate solution of \eqref{lasso}, i.e. $z_{i}^{\sharp} = \beta_{\ell_{i}}$. In the noiseless case, successful recovery would mean that the minimizer of \eqref{lasso} is equal to the candidate solution, i.e. $z_{i}^{\natural} = \beta_{\ell_{i}}$. Even in this case, note that $\|z_{i}^{\natural}-z_{j}^{\natural}\|_{2} \neq 0$ for most pairs $(i,j)$. Nonetheless, we will prove that if the measurements are well-separated in a certain sense, the minimizer of \eqref{lasso} can still be the candidate solution. 
  
We prove that under two technical assumptions, \eqref{lasso} can exactly recover the regression coefficients in the noiseless case. In order to state these assumptions, let
\begin{equation}\label{vpq}
	v_{pq}  = \frac{\beta_{p}-\beta_{q}}{\|\beta_{p}-\beta_{q}\|_{2}}
\end{equation}
 and define the weighted directions, $v_{p}$, by
 \begin{equation}\label{vp}
 	v_{p} := \frac{1}{\sum_{q\neq p} n_{q}} \sum_{q \neq p}n_{q}v_{pq}.
 \end{equation}
That is, $v_{p}$ is the weighted average of the directions to $\beta_{p}$ from other mixture components $\{\beta_{q}\}_{q\neq p}$.

Our first assumption is that the measurements are ``well-separated" in the following sense:
\begin{equation}\label{well-separation}
\max_{p \in [k]}\max_{i \in S_{p}} \frac{\|P_{v_{p}^{\perp}}a_{i}\|_{2}}{\|P_{v_{p}}a_{i}\|_{2}} < \frac{1}{2}\min_{p\in [k]}\frac{n_{p}}{m},	
\end{equation}
where $P_{v_{p}}$ is the projector onto the span of $\{v_{p}\}$ and $P_{v_{p}^{\perp}}$ is the projector onto the subspace orthogonal to $v_{p}$. Intuitively, $\{(a_{i},b_{i})\}_{i =1}^{m}$ is well-separated if $a_{i}$ is approximately parallel to $v_{l_{i}}$. This condition ensures that the hyperplanes corresponding to a fixed class are within a small angle of each other. Note that $\frac{\|P_{v}^{\perp}a\|_{2}}{\|P_{v}a\|_{2}}=0$ if $v$ is parallel to $a$ and is small if $v$ is approximately parallel to $a$. Examples of well-separated data and not well-separated data are shown in Figures \ref{rec} and \ref{fail}.

Our second assumption is that the measurements are ``balanced" in the following sense:
\begin{equation}\label{balance}
\sum_{i\in S_{p}}\sign(v_{p}^{\top}a_{i})\frac{P_{v_{p}^{\perp}}a_{i}}{\|P_{v_{p}}a_{i}\|_{2}} = 0,\text{ for all } p \in {\change \{1,\dots,k\}}. 	
\end{equation}
Intuitively, $\{(a_{i},b_{i})\}_{i =1}^{m}$ is balanced if a particular average of the $\{a_{i}\}_{i\in S_{p}}$ is exactly in the direction of $v_{p}$. Examples of balanced data and not balanced data are shown in Figures \ref{rec} and \ref{fail}.

\begin{figure}[H]
	\begin{subfigure}{0.47\textwidth}
		\captionsetup{skip=-1pt} 
		\psfragscanon
		\psfrag{a}{$\beta_{1}$}
		\psfrag{b}{$\beta_{2}$}
		\psfrag{c}{$\beta_{3}$}
		\psfrag{e}{$v_{1}$}
		\psfrag{f}{$v_{2}$}	
		\psfrag{g}{$v_{3}$}
		\centering
		\includegraphics[scale = .4,center]{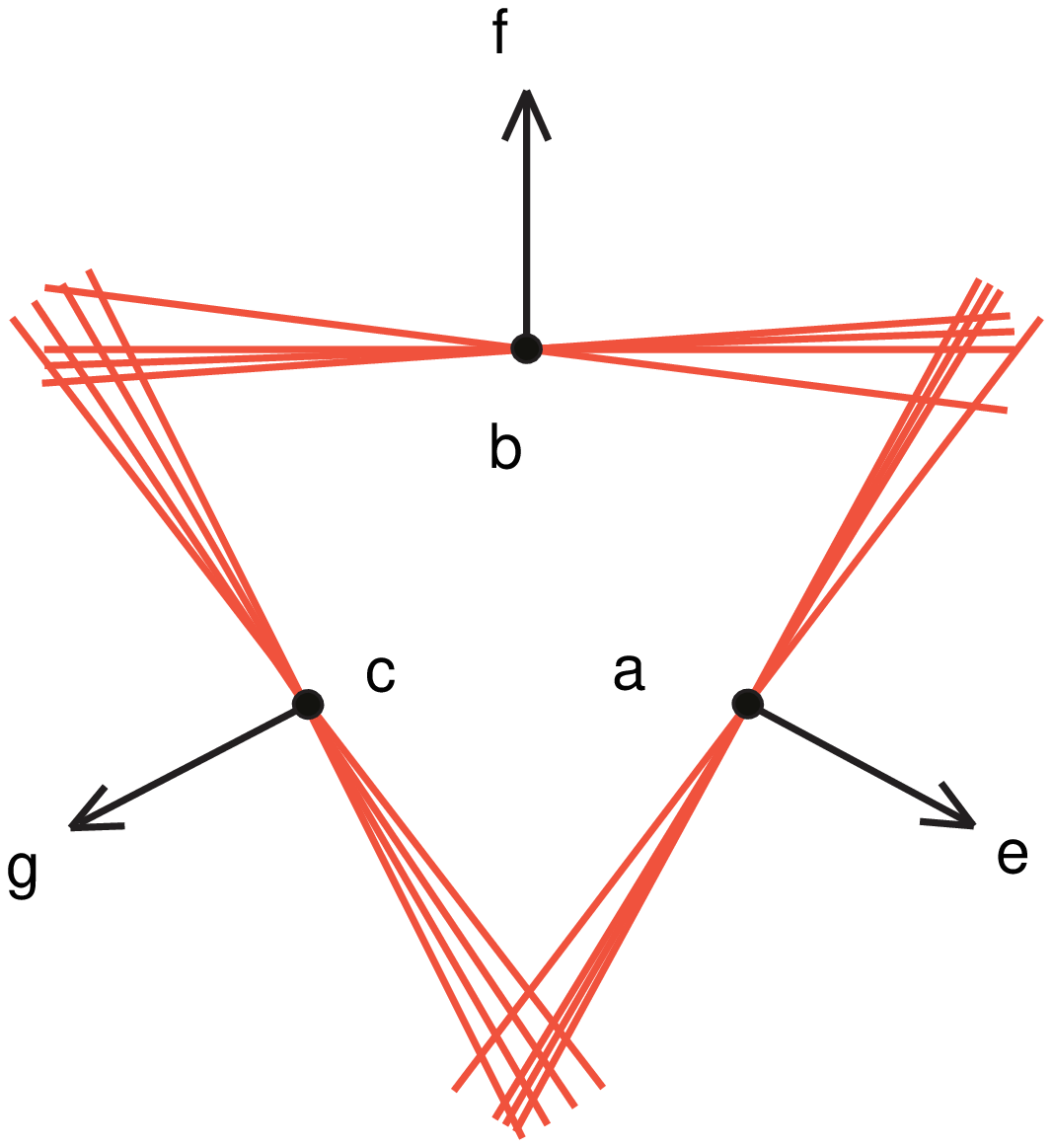} 
		\caption{Well-separated and balanced}
		\label{rec}
	\end{subfigure}
	\quad 
	\begin{subfigure}{0.47 \textwidth}
		\captionsetup{skip=-1pt} 
		\psfragscanon
		\psfrag{a}{$\beta_{1}$}
		\psfrag{b}{$\beta_{2}$}
		\psfrag{c}{$\beta_{3}$}
		\psfrag{e}{$v_{1}$}
		\psfrag{f}{$v_{2}$}
		\psfrag{g}{$v_{3}$}
		\centering
		\includegraphics[scale = .4,center]{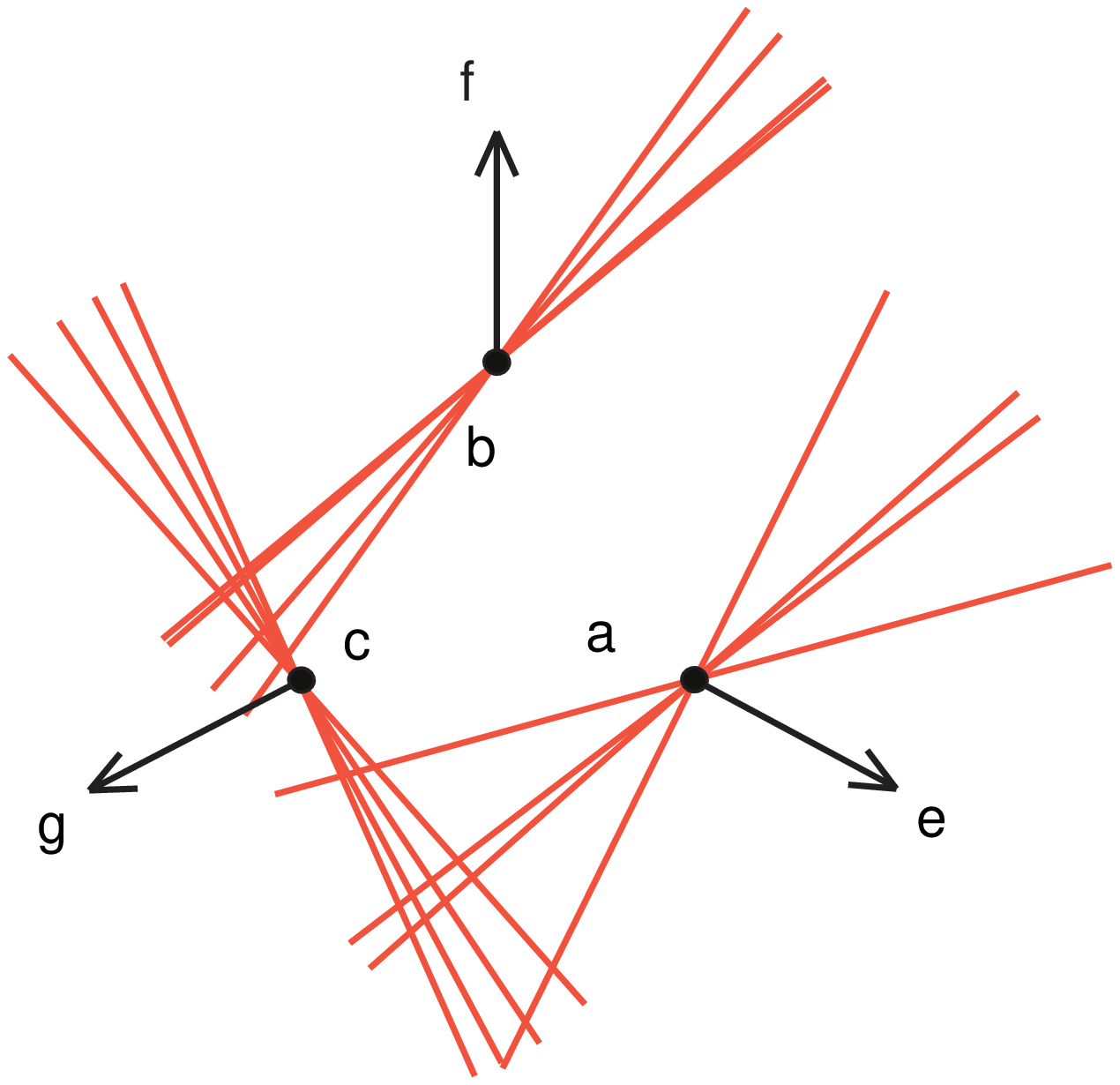} 
		\caption{Not well-separated and not balanced}
		\label{fail}
	\end{subfigure}
	\captionsetup{oneside,margin={0em,0em}}
	\caption{The collection of lines in panel (a) is well-separated in the sense of \eqref{well-separation}. These measurements also satisfy the balance condition \eqref{balance}. {\change In contrast, the collection of lines in panel (b) is not well-separated because the angle between $\nu_{1}$ and a few measurements corresponding to $\beta_{1}$ is too big. Similarly, the collection of lines in panel (b) is not balanced because the weighted average of the normal vectors corresponding to measurements in $\beta_2$ is not in the direction of $\nu_2$.} Note that $v_{p}$ is the weighted average of the directions to $\beta_{p}$ from other mixture components $\{\beta_{q}\}_{q\neq p}$.}
\end{figure}
\vspace{-.3cm}
 
Our main result is that in the noiseless case, \eqref{lasso} uniquely recovers the mixture components with well-separated and balanced measurements, provided there are $d$ independent measurements per class. 
\begin{theorem}\label{thm:recovery}
Fix $\{\beta_{p}\}_{p=1}^{k} \subset \mathbb{R}^{d}$. Let $\{(a_{i},b_{i})\}_{i=1}^{m} \subset \mathbb{R}^{d} \times \mathbb{R}$ be measurements that satisfy (\ref{problem}). If the  measurements satisfy the well-separation condition \eqref{well-separation}, the balance condition \eqref{balance} and for all $p \in [k]$, span$\{a_{i}:i\in S_{p}\} = \mathbb{R}^{d}$, then the solution to \eqref{lasso} is unique and satisfies ${z_{i}} = \beta_{l_{i}}$ for all $i \in [m]$.
\end{theorem}
{\change 
Numerical simulation on synthetic data verify Theorem \ref{thm:recovery}. We present the results of numerical simulation on two types of synthetic data, both satisfying the well-separation property \eqref{well-separation}. In the first simulation, the data satisfies the balance property \eqref{balance} as well. The second simulation shows that recovery is possible in the imbalanced case. We also present the results of numerical simulation on two real datasets and show that the algorithm presented in this paper estimate the regression coefficients and classifies the datasets reasonably well.


\subsection{Discussion and future work}

We propose a convex approach for mixed linear regression with $k$ regression coefficients. An interesting observation in the noiseless case is that the convex program \eqref{lasso} can recover both the sparsity pattern and the regression coefficients exactly. In contrast, consider the Basis Pursuit Denoise (BPDN) program 
\begin{equation}\label{lasso_orginal}
	\min_{x\in \mathbb{R}^{n}} \|x\|_1 + \lambda \|Ax-b\|_2^2	
\end{equation}
in the case where $b = Ax_{o}$, for a sparse $x_o$. For appropriate $\lambda$, the optimizer of \eqref{lasso_orginal} will have the correct sparsity pattern but incorrect estimates of the coefficients of $x_o$, see \cite{Li}. Given that both BPDN and \eqref{lasso} are based on L1 minimization, a natural question is why can \eqref{lasso} recover both sparsity pattern and the estimates? This is because if the sparsity pattern is recovered using \eqref{lasso} in the noiseless case  then $z_i = z_j$ for $i, \ j$ in the same class. Since there are sufficient independent points corresponding to each class, there is only one possible value of $z_i$. In the noisy case of mixed linear regression, an estimation step like clustering using k-means is required. This is analogous to the estimation step required for estimating the regression coefficient from the solution to the BPDN program. 

Another interesting observation is that the program \eqref{lasso} does not have a free parameter when posed in the noiseless case. This is analogous to the Basis Pursuit program 
\begin{equation}\label{BP}
\min_{x\in \mathbb{R}^{n}} \|x\|_1 \hspace{1mm} \text{  subject to  } \hspace{1mm}  Ax=b
\end{equation}
which also does not have any parameter in the noiseless case. 

In this paper, we analyze the noiseless mixed linear regression problem and provide a recovery guarantee when data satisfies conditions \eqref{well-separation} and \eqref{balance}. The current paper is mainly focused in providing an understanding of the well-separation condition while assuming that the data is balanced. However, real data is never exactly balanced. Thus, understanding the  level of imbalance the algorithm can handle is a fruitful area of future work. In the same vein, another important research direction is to consider a more complete model that include noise and corruption. 

}


\subsection{Organization of the paper}
The remainder of the paper is organized as follows. In Section \ref{notation}, we present notations used throughout the paper. In Section \ref{Proof}, we present the proof of Theorem \ref{thm:recovery}. In Section \ref{Simulation}, we introduce an iteratively reweighted least squares implementation that solves \eqref{lasso}. We observe its performance on real and synthetic data.


\subsection{Notation}\label{notation}
Let $[n] = \{1,\dots,n\}$. Let $I_{d\times d}$ be the $d\times d$ identity matrix. Let $B_{r}^{d}$ be the $d$ dimensional {\change Euclidean} ball of radius $r$ centered at the origin {\change and let $S_{r}^{d}$ be the $d$ dimensional Euclidean sphere of radius $r$ centered at the origin}. For any matrix $X$ and any vector $v$, let $\|X\|_{F}$ be the Frobenius norm of $X$ and let $\|v\|_{2}$ be the $\ell_{2}$ norm of $v$. For any non-zero vector $v$, let $\hat{v} = \frac{v}{\ \|v\|_{2}}$. For a vector $v \in \mathbb{R}^{d}$, let $P_{v^{\perp}} = I_{d\times d} - \hat{v}\hat{v}^{\top}$ and $P_{v} = \hat{v}\hat{v}^{\top}$. For a vector $a_{i} \in \mathbb{R}^{d}$, let $a_{is}$ be the $s$th element of $a_i$. Throughout the paper, indices $i$ and $j$ are related to data points and indices $p$ and $q$ are related to mixture components.


\section{Proof}\label{Proof}

We will prove Theorem \ref{thm:recovery} by constructing an exact dual certificate. {\change A dual certificate is a dual optimal point which certifies that the candidate solution is globally optimal, see \cite{Hand,Candes, Fuchs}.} To arrive at the form of the dual certificate we now derive the KKT conditions for \eqref{lasso}. Let $f:(\mathbb{R}^{d})^{m}\rightarrow \mathbb{R}$ be
\[f(Z) = f(z_{1},\dots,z_{m}) = \sum_{i = 1}^{m}\sum_{j=1}^{m} \|z_{i} - z_{j}\|_{2}. \]
The augmented Lagrangian for \eqref{lasso} is
\begin{equation*}
\mathcal{L}(Z,\nu) = 	\sum_{i = 1}^{m}\sum_{j=1}^{m} \|z_{i} - z_{j}\|_{2} - \sum_{i=1}^{m}\nu_{i}\left(a_{i}^{\top}z_{i}-b_{i}\right),
\end{equation*}
where $\nu \in \mathbb{R}^{m}$ is a Lagrange multiplier. Let $Z^{\sharp}$ be the candidate solution, i.e. $z_{i}^{\sharp} = \beta_{l_{i}}$ for all $i \in {\change \{1,\dots,m\}}$. The first order optimality conditions for \eqref{lasso} at $Z^{\sharp}$ are
\begin{align}
& 0 \in \partial_{Z} \left(\sum_{i = 1}^{m}\sum_{j=1}^{m} \|z_{i} - z_{j}\|_{2}\right)(Z^{\sharp}) - \partial_{Z}\left(\sum_{i=1}^{m}\nu_{i}\left(a_{i}^{\top}z_{i}-b_{i}\right)\right)(Z^{\sharp}),\\
&	a_{i}^{\ T}z_{i}^{\sharp} = b_{i},\ i\in {\change \{1,\dots,m\}},
\end{align}
where  $\partial_{Z} f(Z^{\sharp})$ is the subdifferential of $f$ with respect to $Z$ evaluated at $Z^{\sharp}$. Note that 
\begin{equation}
	\partial_{(z_{i},z_{j})} \|z_{i}-z_{j}\|_{2}(z_{i}^{\sharp},z_{j}^{\sharp})
	= \left\{ \begin{array}{ll}
 				\left\{\left[\begin{array}{c}
								v_{pq}\\
								-v_{pq}	
								\end{array} \right]\right\}, & \ell_{i} = p,\ \ell_{j} = q,\ p\neq q\\
				\left\{\left[\begin{array}{c}
								\xi_{ij}\\
								-\xi_{ij}	
								\end{array}\right]: \|\xi_{ij}\|_{2}\leq 1 \right\}, & z_{i}^{\sharp} = z_{j}^{\sharp}
				 \end{array}\right.
\end{equation}
where $v_{pq}$ is defined in \eqref{vpq} and $\xi_{ij} \in \mathbb{R}^{d}$. So, the KKT conditions for \eqref{lasso} are: for all $p \in {\change \{1,\dots,k\}}$ and $i \in S_{p}$,
\begin{equation}
0 \in \left\{\sum_{j\in S_{p}, j\neq i} \xi_{ij} + \sum_{q\neq p}n_{q}v_{pq} - \nu_{i}a_{i}: \xi_{ij} \in \mathbb{R}^{d},\ \|\xi_{ij}\|_{2}\leq 1,\ \xi_{ij} = -\xi_{ji},\ \nu_{i} \in \mathbb{R}\right\} 
\end{equation}

We will now prove that if such $\xi_{ij}$ and $\nu$ exist and if all $\|\xi_{ij}\|_{2}<1$, then the unique output of \eqref{lasso} is the candidate solution $Z^{\sharp}$.
\begin{lemma} \label{exact}  Fix $\{\beta_{p}\}_{p=1}^{k} \subset \mathbb{R}^{d}$. Let $\{(a_{i},b_{i})\}_{i=1}^{m} \subset \mathbb{R}^{d} \times \mathbb{R}$ be measurements that satisfy \eqref{problem}. Assume for $p \in {\change \{1,\dots,k\}}$ and for all $i,j \in S_{p}$ , there exists $\xi_{ij} \in \mathbb{R}^{d}$ and $\nu_{i} \in \mathbb{R}$ such that
\begin{align}
	&\nu_{i}a_{i}=\sum_{j\in S_{p}, j\neq i} \xi_{ij} + \sum_{q\neq p}n_{q}v_{p}  ,  \label{S1}\\
     &\|\xi_{ij}\|_{2} < 1 \label{S2}, \\
	 &\xi_{ij} = -\xi_{ji} \label{S3}. 
\end{align}
Also, assume for all $p \in {\change \{1,\dots,k\}}$, span$\{a_{i}:i\in S_{p}\}=\mathbb{R}^{d}$. Then, $Z^{\sharp}$ is the unique solution to \eqref{lasso}.
\end{lemma}

\begin{proof} We will show that $f(Z^{\sharp} + H) > f(Z^{\sharp})$ for any feasible perturbation $H\neq 0$. We first separate $f(Z^{\sharp}+H)$ into two parts.
\begin{equation}\label{split}
	f(Z^{\sharp}+H) = \sum_{p=1}^{k}\sum_{i,j \in S_{p}}\|h_{i}-h_{j}\|_{2} + \underbrace{\sum_{p=1}^{k}\sum_{q\neq p}\sum_{i\in S_{p}}\sum_{j \in S_{q}}\|z_{i}^{\sharp} -z_{j}^{\sharp}+h_{i}-h_{j}\|_{2}}_{\rom{1}} 
\end{equation}
To bound $\rom{1}$ from below, note that $\|z_{i}^{\sharp} -z_{j}^{\sharp}+h_{i}-h_{j}\|_{2} \geq \|z_{i}^{\sharp} -z_{j}^{\sharp}\|_{2}+v_{pq}^{\top	}(h_{i}-h_{j})$ because $v_{pq}$, defined in \eqref{vpq}, is a subgradient of $\|\cdot\|_{2}$ at $z_{i}^{\sharp}-z_{j}^{\sharp}$ if $i \in S_{p}$ and $j \in S_{q}$. Thus, 
\begin{align}
	\rom{1} = & \sum_{p=1}^{k}\sum_{q\neq p}\sum_{i\in S_{p}}\sum_{j \in S_{q}}\|z_{i}^{\sharp} -z_{j}^{\sharp}+h_{i}-h_{j}\|_{2}\\
	 \geq & \sum_{p=1}^{k}\sum_{q\neq p}\sum_{i\in S_{p}}\sum_{j \in S_{q}}\|z_{i}^{\sharp} - z_{j}^{\sharp}\|_{2} + \sum_{p=1}^{k}\sum_{q\neq p}\sum_{i\in S_{p}}\sum_{j \in S_{q}}v_{pq}^{\top}(h_{i}-h_{j})\\
	 = & f(Z^{\sharp})+ \sum_{p=1}^{k}\sum_{q\neq p}\sum_{i\in S_{p}}\sum_{j \in S_{q}}v_{pq}^{\top}h_{i} - \sum_{p=1}^{k}\sum_{q\neq p}\sum_{i\in S_{p}}\sum_{j \in S_{q}}v_{pq}^{\top}h_{j}\\
	= & f(Z^{\sharp})+\sum_{p=1}^{k}\sum_{q\neq p}\sum_{i \in S_{p}}n_{q}v_{pq}^{\top}h_{i}+\sum_{p=1}^{k}\sum_{q\neq p}\sum_{j\in S_{q}}n_{p}v_{qp}^{\top}h_{j} \label{substi}\\ 
	= & f(Z^{\sharp})+2\sum_{p=1}^{k}\sum_{i \in S_{p}}\sum_{q\neq p}n_{q}v_{p}^{\top}h_{i}\label{def}.
\end{align}
Note that in \eqref{substi} we used $v_{pq} = -v_{qp}$ and in \eqref{def} we used \eqref{vp}. Combining \eqref{split} and \eqref{def}, it suffices to show   
\begin{equation}\label{perturbation}
	\sum_{p=1}^{k}\sum_{i,j \in S_{p}}\|h_{i}-h_{j}\|_{2} + 2\underbrace{\sum_{p=1}^{k}\sum_{i\in S_{p}}\sum_{q\neq p}n_{q}v_{p}^{\top}h_{i}}_{\rom{2}} >0.
\end{equation}
for all feasible $H \neq 0$.

First, we show that for all feasible $H \neq 0$, there exists $p \in {\change \{1,\dots,k\}}$ and $i,j \in S_{p}$ such that $h_{i} \neq h_{j}$. We provide a proof by contradiction. Assume for all $p \in {\change \{1,\dots,k\}}$ and $i,j \in S_{p}$, $h_{i} = h_{j}$.  Fix $p\in {\change \{1,\dots,k\}}$. Let $h_{j} = c_{p}$ for all $j \in S_{p}$. Since $c_{p}$ is a feasible perturbation, $a_{j}^{\top}c_{p} = 0$ for all $j \in S_{p}$. Thus, $c_{p}$ is orthogonal to any element in span$\{a_{j}:j\in S_{p}\} = \mathbb{R}^{d}$. Hence, $h_{j} = c_{p} = 0$ for all $j \in S_{p}$, which contradicts $H \neq 0$.

We now compute
\begin{align}
\rom{2} = & \sum_{p=1}^{k}\sum_{i \in S_{p}}\bigg(\sum_{q\neq p}n_{q}v_{p}\bigg)^{\top}h_{i} \\
= & \sum_{p=1}^{k}\sum_{i \in S_{p}}\bigg(\nu_{i}a_{i} - \sum_{j\in S_{p},j\neq i}\xi_{ij}\bigg)^{\top}h_{i}\label{sub2}\\
= & -\frac{1}{2}\sum_{p=1}^{k}2\sum_{i \in S_{p}}\sum_{j\in S_{p},j\neq i}\xi_{ij}^{\top}h_{i}\label{feas}\\
= & -\frac{1}{2}\sum_{p=1}^{k}\bigg(\sum_{i \in S_{p}}\sum_{j\in S_{p},j\neq i}\xi_{ij}^{\top}h_{i} + \sum_{j \in S_{p}}\sum_{i\in S_{p},i\neq j}\xi_{ji}^{\top}h_{j}\bigg) \label{expand}\\
= & -\frac{1}{2}\sum_{p=1}^{k}\sum_{i \in S_{p}}\sum_{j\in S_{p},j\neq i}\xi_{ij}^{\top}\left(h_{i} -h_{j}\right)\label{rearrange}\\
\geq & -\frac{1}{2} \sum_{p=1}^{k}\sum_{i \in S_{p}}\sum_{j\in S_{p}}\|\xi_{ij}\|_{2}\|h_{i}-h_{j}\|_{2}\label{Cauchy}\\
= & -\frac{\gamma}{2} \sum_{p=1}^{k}\sum_{i \in S_{p}}\sum_{j\in S_{p}}\|h_{i}-h_{j}\|_{2},\label{gamma}
\end{align}
where $\gamma:=\max_{p \in [k]}\max_{i,j \in S_{\change p}}\|\xi_{ij}\|_{2}$. By assumption \eqref{S2}, $\gamma < 1$. Note that in \eqref{sub2}, we used \eqref{S1}. In \eqref{feas}, we used $a_{i}^{\top}h_{i} = 0$ (since $H$ is feasible). In \eqref{expand}, both terms in parenthesis are equal by interchanging the dummy variables $i$ and $j$. In \eqref{rearrange}, we used the antisymmetry condition \eqref{S3}. In \eqref{Cauchy}, we used Cauchy-Schwartz inequality.  Combining \eqref{perturbation} and \eqref{gamma}, we get
\begin{align}
	\sum_{p=1}^{k}\bigg(\sum_{i,j \in S_{p}}\|h_{i}-h_{j}\|_{2} + 2\sum_{i\in S_{p}}\sum_{q\neq p}n_{q}v_{p}^{\top}h_{i}\bigg) &\geq (1-\gamma)\sum_{p=1}^{k}\sum_{i,j\in S_{p}}\|h_{p}-h_{q}\|_{2} > 0.\label{hneq0}
\end{align}
The strict inequality in \eqref{hneq0} holds because $(1-\gamma) > 0$ and $H\neq 0$ implies that there exists $p \in {\change \{1,\dots,k\}}$ and $i,j \in S_{p}$ such that $h_{i} \neq h_{j}$. Hence, $Z^{\sharp}$ is the unique solution to \eqref{lasso}
\end{proof}

As a consequence of Lemma \ref{exact}, the proof of Theorem \ref{thm:recovery} is simplified to constructing an exact certificate that satisfies \eqref{S1}, \eqref{S2} and \eqref{S3}.

\begin{proof}[Proof of Theorem \ref{thm:recovery}] For all $p \in {\change \{1,\dots,k\}}$ and $i,j \in S_{p}$, let $\xi_{ij}$ and $\nu_{i}$ be
\begin{align}
	\nu_{i} &= \sign(v_{p}^{\top}a_{i})\frac{\|v_{p}\|_{2}\sum_{q\neq p}n_{q}}{\|P_{v_{p}}a_{i}\|_{2}} \label{nu}\\
	\xi_{ij} &= \frac{1}{n_{p}}\left(\nu_{i}P_{v_{p}^{\perp}}a_{i} - \nu_{j} P_{v_{p}^{\perp}}a_{j}\right). \label{xi}	
\end{align}
By Lemma \ref{exact}, it is sufficient to show that, for all $p \in {\change \{1,\dots,k\}}$ and $i,j \in S_{p}$, these $\xi_{ij}$ and $\nu_{i}$ satisfy \eqref{S1}, \eqref{S2}, \eqref{S3}. Note that \eqref{S3} follows immediately. Condition \eqref{S2} holds because
\begin{align}
\|\xi_{ij}\|_{2} \leq & \frac{1}{n_{p}}\big(\|\nu_{i}P_{v_{p}^{\perp}}a_{i}\|_{2} + \|\nu_{j} P_{v_{p}^{\perp}}a_{j}\|_{2}\big)\label{triangle2} \\
 \leq & \frac{2}{n_{p}} \max_{i \in S_{p}} \|\nu_{i}P_{v_{p}^{\perp}}a_{i}\|_{2}\\
  = & \frac{2}{n_{p}} \max_{i \in S_{p}}\|v_{p}\|_{2}\sum_{q\neq p}n_{q}\frac{\|P_{v_{p}^{\perp}}a_{i}\|_{2}}{\|P_{v_{p}}a_{i}\|_{2}}\label{subst2}\\
  < &\frac{2}{n_{p}}m\frac{n_{p}}{2m}\label{well-sep}\\
  = &\ 1
\end{align}
In \eqref{subst2}, we used \eqref{nu}, and in \eqref{well-sep}, we used \eqref{well-separation} along with $\|v_{p}\|_{2} \leq 1$.

Lastly, we will show that $\eqref{S1}$ holds through direct computation. Fix $p \in {\change \{1,\dots,k\}}$ and $i \in S_{p}$. We now compute
\begin{flalign}
& \sum_{j\in S_{p}, j\neq i} \xi_{ij} + \sum_{q\neq p}n_{q}v_{p} &\\
= & \sum_{j\in S_{p}, j\neq i} \frac{1}{n_{p}}\bigg(\nu_{i}P_{v_{p}^{\perp}}a_{i} - \nu_{j} P_{v_{p}^{\perp}}a_{j}\bigg) + \sum_{q\neq p}n_{q}v_{p}&\\
= & \sum_{j\in S_{p}, j\neq i} \frac{1}{n_{p}}\|v_{p}\|_{2}\bigg(\sum_{q\neq p}n_{q}\bigg)\bigg(\sign(v_{p}^{\top}a_{i})\frac{P_{v_{p}^{\perp}}a_{i}}{\|P_{v_{p}}a_{i}\|_{2}} - \sign(v_{p}^{\top}a_{j})\frac{P_{v_{p}^{\perp}}a_{j}}{\|P_{v_{p}}a_{j}\|_{2}}\bigg) + \sum_{q\neq p}n_{q}v_{p}& \label{subst3}\\
= & \bigg(\sum_{q\neq p}n_{q}\bigg)\frac{\|v_{p}\|_{2}}{n_{p}}\sum_{j\in S_{p}}\bigg(\sign(v_{p}^{\top}a_{i})\frac{P_{v_{p}^{\perp}}a_{i}}{\|P_{v_{p}}a_{i}\|_{2}} - \sign(v_{p}^{\top}a_{j})\frac{P_{v_{p}^{\perp}}a_{j}}{\|P_{v_{p}}a_{j}\|_{2}}\bigg)+\sum_{q\neq p}n_{q}v_{p}&\\
= & \bigg(\sum_{q\neq p}n_{q}\bigg)\frac{\|v_{p}\|_{2}}{n_{p}}\bigg(\sign(v_{p}^{\top}a_{i})n_{p}\frac{P_{v_{p}^{\perp}}a_{i}}{\|P_{v_{p}}a_{i}\|_{2}} - \sum_{j\in S_{p}}\sign(v_{p}^{\top}a_{j})\frac{P_{v_{p}^{\perp}}a_{j}}{\|P_{v_{p}}a_{j}\|_{2}}\bigg)+\bigg(\sum_{q\neq p}n_{q}\bigg)v_{p}& \\
= & \bigg(\sum_{q\neq p}n_{q}\bigg)\|v_{p}\|_{2}\left(\sign(v_{p}^{\top}a_{i})\frac{P_{v_{p}^{\perp}}a_{i}}{\|P_{v_{p}}a_{i}\|_{2}}+\frac{v_{p}}{\|v_{p}\|_{2}}\right)&\label{balanceass}\\
= & \bigg(\sum_{q\neq p}n_{q}\bigg)\frac{\|v_{p}\|_{2}\sign(v_{p}^{\top}a_{i})}{\|P_{v_{p}}a_{i}\|_{2}}\bigg(P_{v_{p}^{\perp}}a_{i}+\sign(v_{p}^{\top}a_{i})\frac{v_{p}}{\|v_{p}\|_{2}}\|P_{v_{p}}a_{i}\|_{2}\bigg)&\\
= & \nu_{i}\bigg(P_{v_{p}^{\perp}}a_{i}+P_{v_{p}}a_{i}\bigg) &\label{Pva}\\
= & \nu_{i}a_{i}. &
\end{flalign}
 Note that in \eqref{subst3}, we used \eqref{nu}. In \eqref{balanceass}, we used the assumed balance condition \eqref{balance}, and in \eqref{Pva}, we used \eqref{nu} and $P_{v_{p}}a_{i} = \sign(v_{p}^{\top}a_{i})\|P_{v_{p}}a_{i}\|_{2}\frac{v_{p}}{\|v_{p}\|_{2}}$. Hence, \eqref{S1} holds. By Lemma \ref{exact}, $Z^{\sharp}$ is the unique solution to \eqref{lasso}. 
\end{proof}


\subsection{Derivation of Dual Certificate}

In this section, we provide a derivation of constructing the dual certificate. Fix $p \in {\change \{1,\dots,k\}}$. Without loss of generality, assume $S_{p} = {\change \{1,\dots,n_{p}\}}$.

Note that \eqref{S1} can be decomposed into its components along and orthogonal to $v_{p}$:
\begin{equation}\label{alongv}
	\sum_{j\in S_{p}, j\neq i} P_{v_{p}}\xi_{ij} + \sum_{q\neq p}n_{q}v_{p} = \nu_{i}P_{v_{p}}a_{i}, \ i \in S_{p}
\end{equation}
\begin{equation} \label{orthov}
	\sum_{j\in S_{p}, j\neq i} P_{v_{p}^{\perp}} \xi_{ij} = \nu_{i}P_{v_{p}^{\perp}} a_{i}, \ i \in S_{p}.
\end{equation} 

Note that solving (\ref{orthov}) is equivalent to solving a system of linear equations $A\xi = b$, where $A \in \mathbb{R}^{n_{p}d \times \frac{n_{p}(n_{p}-1)d}{2}}$ and $b\in \mathbb{R}^{n_{p}d}$. Let $\tilde{B}$ be the first ($n_{p}-1$) block rows of $A$ and $\tilde{b}$ be the first ($n_{p} - 1$) blocks of $b$. So, the $i$th block of $\tilde{b}$ is $\nu_{i}P_{v_{p}^{\perp}}a_{i}$ and the matrix $\tilde{B}$ and the vector $\xi$ are

\begin{equation*}
(\tilde{B})_{il} = \left\{\begin{array}{l l}
						I_{d\times d} & \text{ if } l \in \big\{\sum_{r=1}^{i-1}(n_{p}-r) + s: \ s = 1,\dots, n_{p}-i\big\}\\
						-I_{d\times d} & \text{ if } i >1, \ l \in \big\{ \sum_{r=1}^{i-s}(n_{p}-r)+i-s+1: \ s = 2,\dots, i\big\}\\
						0 & \text{ otherwise,}
						\end{array}	
					\right.
\end{equation*}
\begin{equation*}
(\xi)_{l} = P_{v_{p}^{\perp}}\xi_{ij} \text{ if } i = \argmax_{\tilde{i}\in[n_{p}-1]}\sum_{s = 1}^{\tilde{i}}\frac{l-\sum_{r=1}^{s-1}(n_{p}-r)}{\left|l-\sum_{r=1}^{s-1}(n_{p}-r)\right|},\ j = l-\sum_{r=1}^{i-1}(n_{p}-r)+i.
\end{equation*}

It is straightforward to verify that $\tilde{B}$ is full row rank and the least squares solution of $\tilde{B}\xi = \tilde{b}$ is also a solution to $A\xi= b$ if
\begin{equation} \label{ls}
	\sum_{i = 1}^{n_{p}} \nu_{i}P_{v_{p}^{\perp}}a_{i} = 0.
\end{equation}
Now, for all $j\in S_{p}$, choose $P_{v_{p}}\xi_{ij} = 0$ in \eqref{alongv}. Then, for all $i \in S_{p}$,
\begin{equation}
	\nu_{i} = \sign(v_{p}^{\top}a_{i})\frac{\|v_{p}\|_{2}\sum_{q\neq p}n_{q}}{\|P_{v_{p}}a_{i}\|_{2}}.
\end{equation}
Using these $\nu_{i}$, \eqref{ls} is satisfied because of the assumed balance condition \eqref{balance}. We note that the least 2-norm solution of $\tilde{B}\xi = \tilde{b}$ is $\xi = \tilde{B}^{T}(\tilde{B}\tilde{B}^{T})^{-1}\tilde{b}$, where
\[(\tilde{B}\tilde{B}^{T})_{ls} = \left\{\begin{array}{cc}
(n_{p} -1)I_{d\times d} & \text{if } l =s\\[.1em]
-I_{d\times d} & \text{if } l \neq s,
\end{array}
\right. \text{ and}
\]
\[
((\tilde{B}\tilde{B}^{T})^{-1})_{ls} = \left\{\begin{array}{cc}
\frac{2}{n_{p}}I_{d\times d} & \text{if } l = s\\[.5em] 
\frac{1}{n_{p}}I_{d\times d} & \text{if } l \neq s.
\end{array}
\right.
\]
Here, $\tilde{B}\tilde{B}^{T} \in \mathbb{R}^{d(n_{j}-1)\times d(n_{j}-1)}$. Solving for $\xi$, we get $P_{v_{p}^{\perp}}\xi_{ij} = \frac{\nu_{i}P_{v_{p}^{\perp}}a_{i}}{n_{p}} - \frac{\nu_{j}P_{v_{p}^{\perp}}a_{j}}{n_{p}}$. Since $P_{v_{p}}\xi_{ij} = 0$, we have $\xi_{ij} = \frac{\nu_{i}P_{v_{p}^{\perp}}a_{i}}{n_{p}} - \frac{\nu_{j}P_{v_{p}^{\perp}}a_{j}}{n_{p}}$.


\section{Numerical Results}
\label{Simulation}
In this section we formulate an iteratively reweighted least squares (IRLS) algorithm for \eqref{lasso} and provide numerical results on real and synthetic data. {\change The general idea of the IRLS algorithm is to approximate a non-smooth objective function, as in \eqref{lasso}, with a sequence of smooth functionals that converges to the objective function. In our implementation of the IRLS algorithm for \eqref{lasso}, the smooth approximation results in a weighted least squares problem which can be solved through the normal equations.} The pseudocode is given in Algorithm \ref{alg:IRLS}. We refer the readers to \cite{Bissantz} for convergence result of the IRLS algorithm.

Algorithm \ref{alg:IRLS} outputs estimated mixture components for each data point. These estimates can then be clustered using $k$-means. The mixture component corresponding to each class is obtained by running $k$ separate regressions.

 When implementing the IRLS algorithm, we observe convergence for a fixed $\delta \ll 1$ ({\change $\delta_t$ in Algorithm \ref{alg:IRLS}} is fixed to $10^{-16}$). The maximum number of iterations is 150. Let $Z_{t}^{\natural}$ be the minimizer of the quadratic program in the IRLS algorithm at iteration $t$. The stopping {\change criterion} is $\frac{1}{\sqrt{m}}\|Z_{t+1}^{\natural} - Z_{t}^{\natural}\|_{F} < 10^{-5}$.
 
\vspace{-6mm}
\begin{Algorithm}[H]
\caption{Iteratively Reweighted Least Squares for \eqref{lasso}}
\label{alg:IRLS}
\begin{algorithmic}[1]
\Statex Input: Measurements $\{(a_{i},b_{i})\}_{i=1}^{m}$,  regularization parameter $\delta_{t} \shortarrow{7} 0$, maximum iteration number $t_{max}$, $w_{ij}^{(0)} = 1$.
\Statex Output: Estimates $\{z_{1},\dots, z_{m}\}$
\While {$ \ t \leq t_{max}$ or convergence }:
\State ${z_{1}^{(t)},\dots,z_{m}^{(t)}} \leftarrow \argmin_{z_{1},\dots,z_{m}} \sum_{i=1}^{m}\sum_{j=1}^{m} w_{ij}^{(t)}\|z_{i}-z_{j}\|_{2}^{2} \ \ \text{s.t.} \ \ a_{i}^{\top}z_{i} = b_{i}, \ i\in {\change \{1,\dots,m\}}$
\State $w_{ij}^{(t+1)} \leftarrow \big(\|z_{i}^{(t)}-z_{j}^{(t)}\|_{2}^{2} + \delta_{t} \big)^{-\frac{1}{2}}$
\EndWhile
\end{algorithmic}
\end{Algorithm}
\vspace{-.6mm}

We use Algorithm \ref{alg:IRLS}, along with $k$-means, to classify two real datasets. {\change In both of these datasets, the true labels of the datapoints are unknown. We compare the output of Algorithm \ref{alg:IRLS} followed by k-means with the output of an algorithm introduced in \cite{Chen1}. The algorithm introduced in \cite{Chen1} recovers an estimate of the mixture components when used on a mixed linear regression problem with bounded but arbitrary noise. For ease of reference, this algorithm is presented in Algorithm \ref{alg:Rank}.}
\vspace{-6mm}
\floatname{algorithm}{\change Algorithm}
\begin{Algorithm}[H]
\caption{\change Low-Rank Matrix Completion for Mixed Linear Regression}\label{alg:Rank}
\change
\begin{algorithmic}[1]
\Statex Input: Measurements $\{(a_{i},b_{i})\}_{i=1}^{m}$, noise  parameter $\eta >0$. 
\Statex Output: Estimates $\{\beta_{1} ,\beta_{2}\}$

\State ${K,g} \leftarrow \argmin_{K,g} \|K\|_{*} \ \ \text{s.t.} \ \ \sum_{i=1}^{m}\left|-\langle a_{i}a_{i}^\top,K\rangle + 2b_{i}\langle a_{i},g\rangle - b_{i}^2 \right|\leq \eta. $
\State $\lambda$, $v \leftarrow$ \text{first eigenvalue-eigenvector pair of } $gg^\top-K$.
\State $(\beta_{1},\beta_2) \leftarrow (g-\sqrt{\lambda}v, g+\sqrt{\lambda}v)$.
\end{algorithmic}
\end{Algorithm}
\vspace{-2mm}

\begin{figure}[H]
  \begin{tabular}[c]{cc}
    \begin{subfigure}[c]{0.45\textwidth}
      \includegraphics[scale = .4]{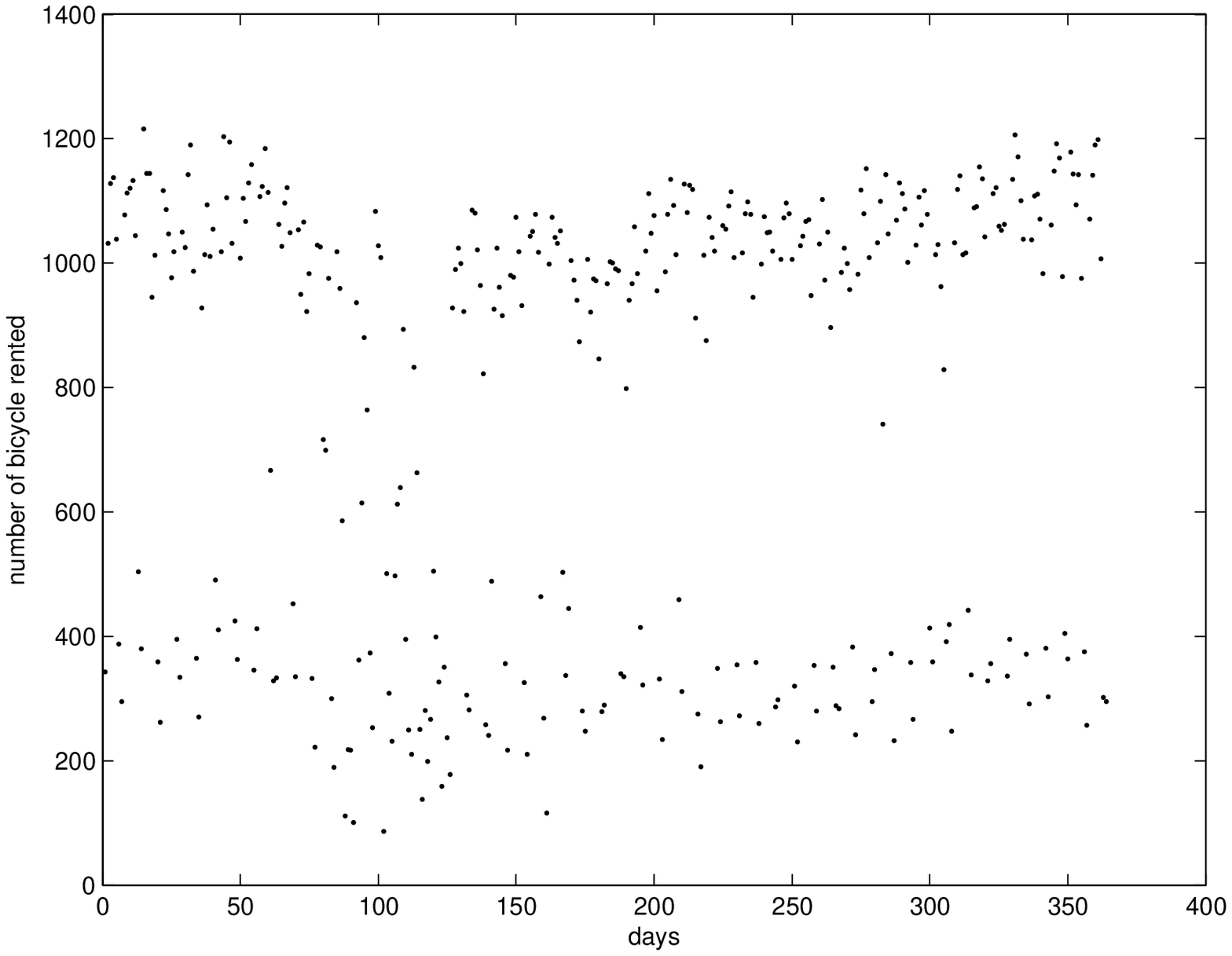}
      \caption{}
      \label{b_datapoints}
    \end{subfigure}&
    \begin{subfigure}[c]{0.45\textwidth}
      \includegraphics[scale = .4]{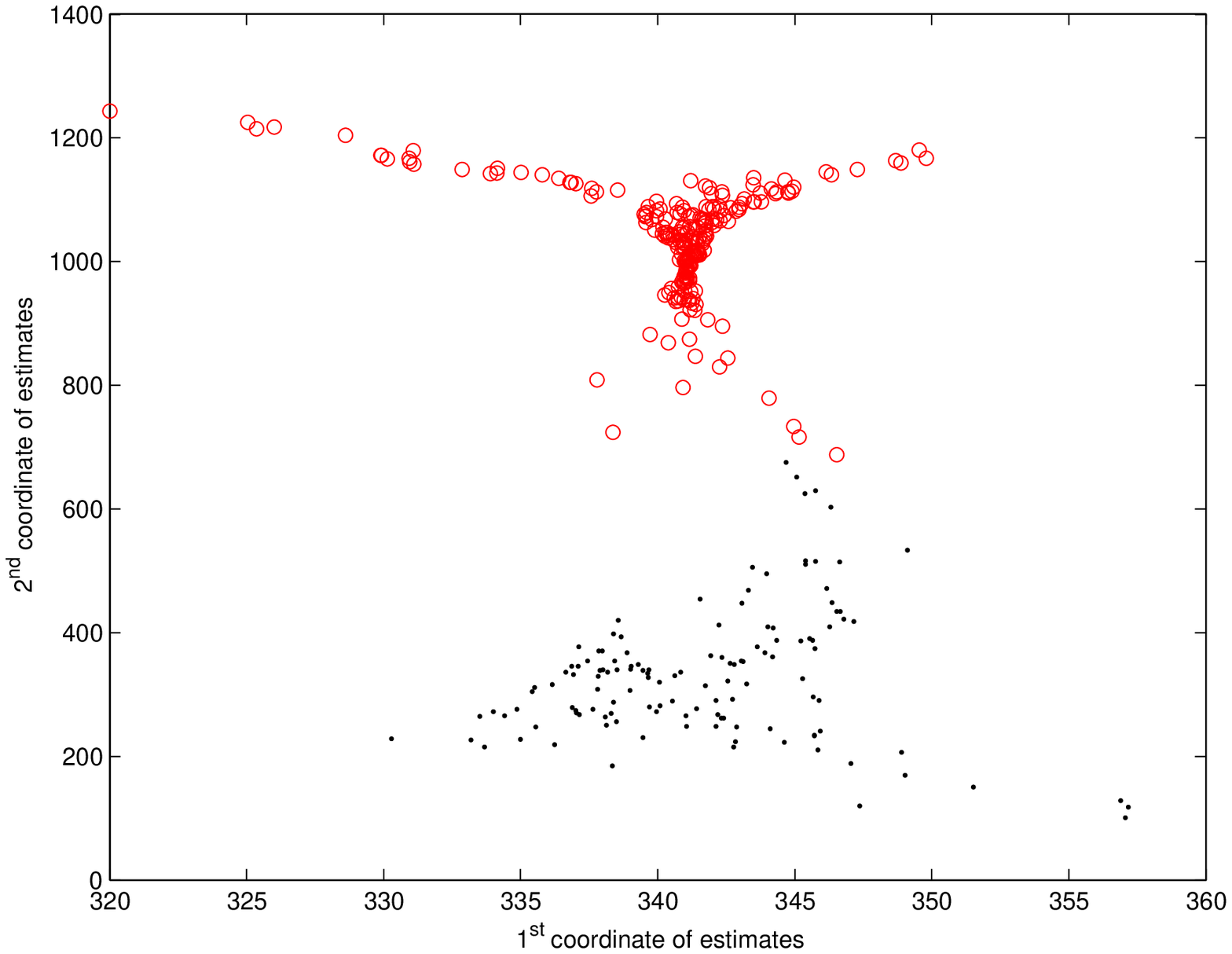}
      \caption{}
      \label{b_estimates}
    \end{subfigure}\\
    \begin{subfigure}[c]{0.45\textwidth}
      \includegraphics[scale = .4]{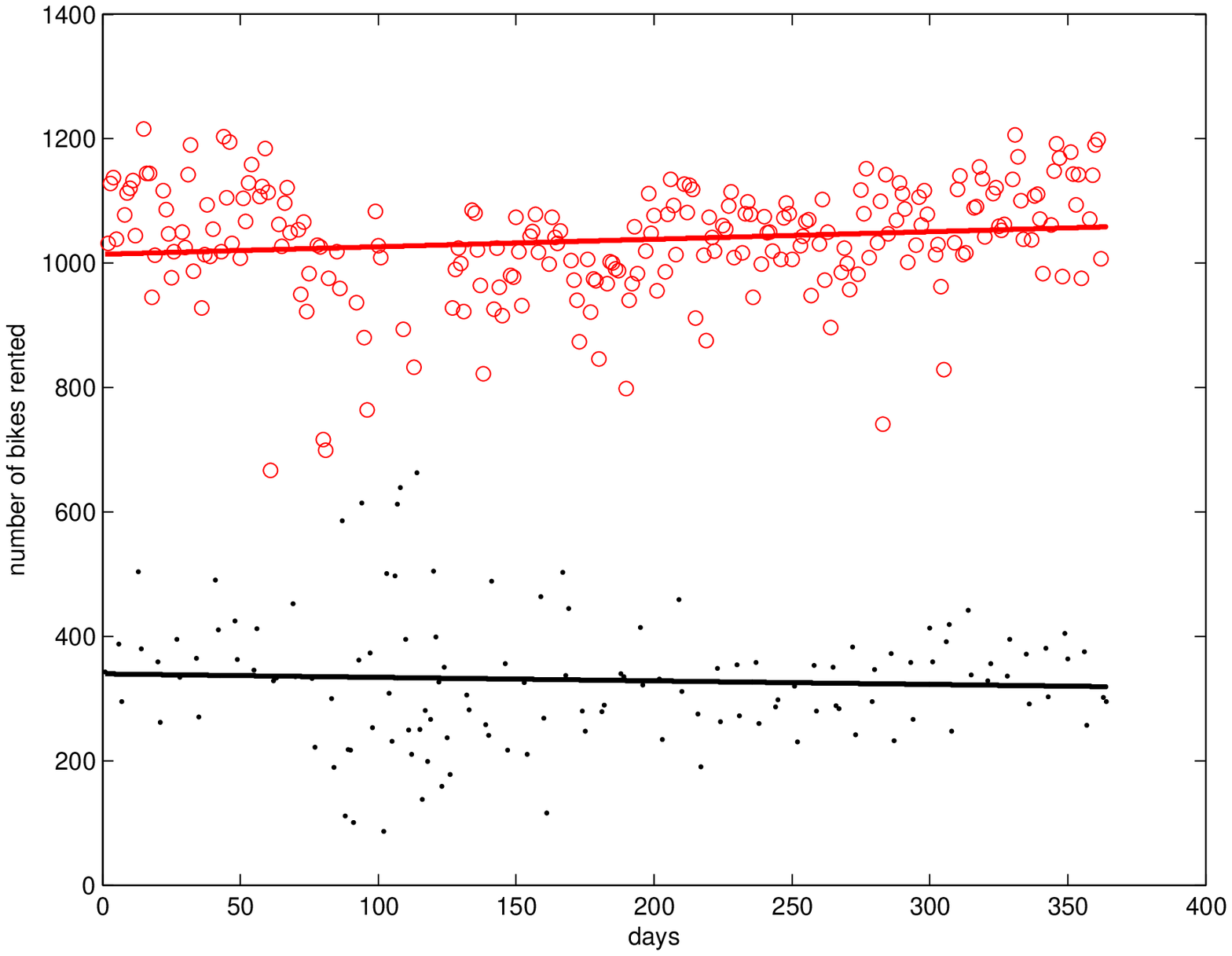}
      \caption{}
		\label{b_fit}
    \end{subfigure}&
    \begin{subfigure}[c]{0.45\textwidth}
      \includegraphics[scale = .4]{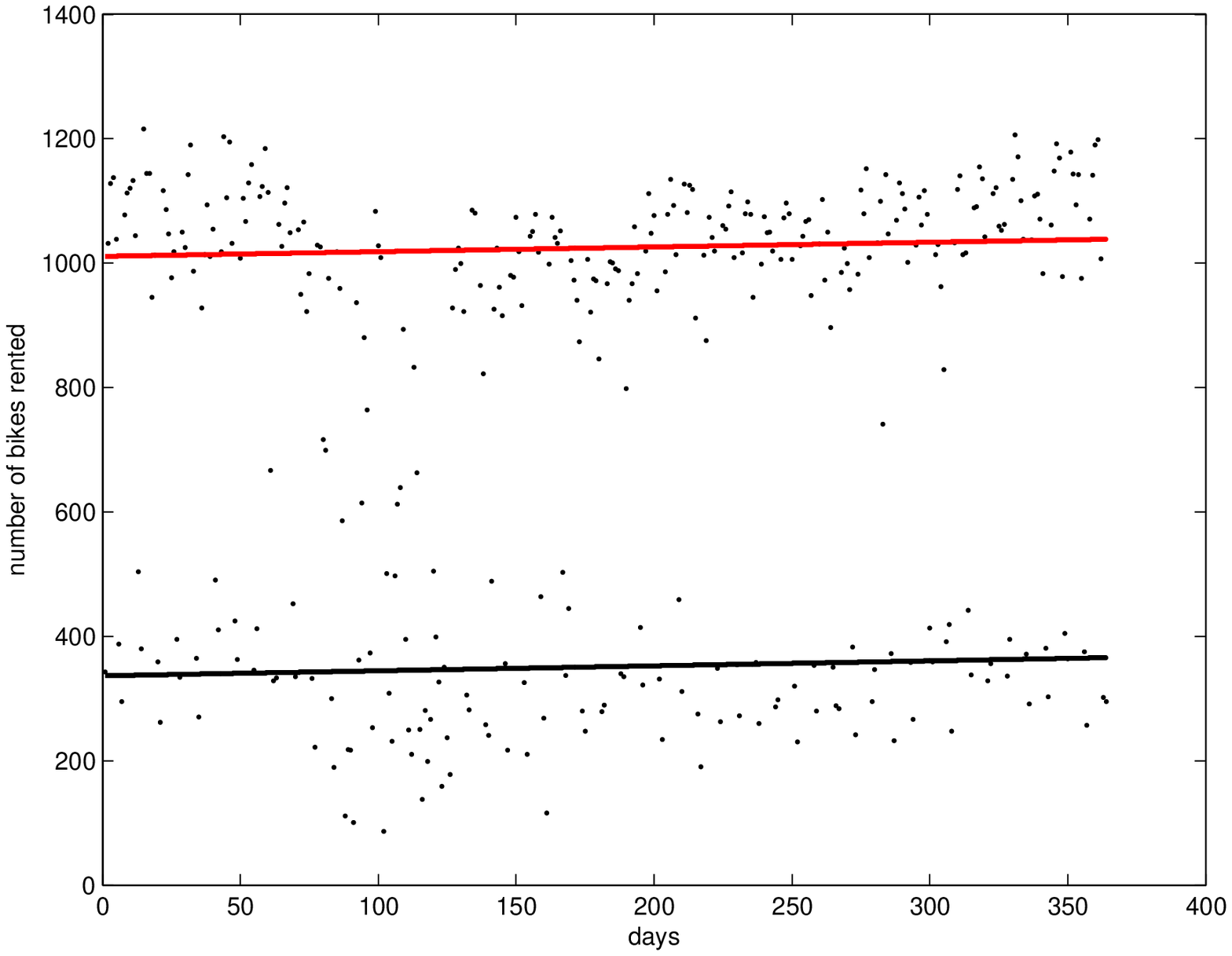}
      \caption{}
		\label{b_fit_lift}
    \end{subfigure}\\
  \end{tabular} 
 \captionsetup{oneside,margin={0em,0em}}
  \caption{Panel (a) shows data consisting of the number of bikes rented per day in the Bay Area BikeShare program. {\change A circle or a dot in panel (b) is an estimate of one of the mixture components corresponding to a data point in panel (a). Panel (b) shows the output of \eqref{lasso} followed by $k$-means with $k = 2$. Panel (c) shows the fitted lines obtained using two separate regressions. In panel (c), points in first and second classes are represented by dots and circles, respectively. Qualitatively, the classification presented in panel (c) seems to differentiate between weekend and weekday bike rental trends. Panel (d) shows the result of Algorithm \ref{alg:Rank} with $\eta = 100$.}}
\end{figure}

\begin{figure}[H]
  \begin{tabular}[c]{cc}
    \begin{subfigure}[c]{0.45\textwidth}
      \includegraphics[scale = .4]{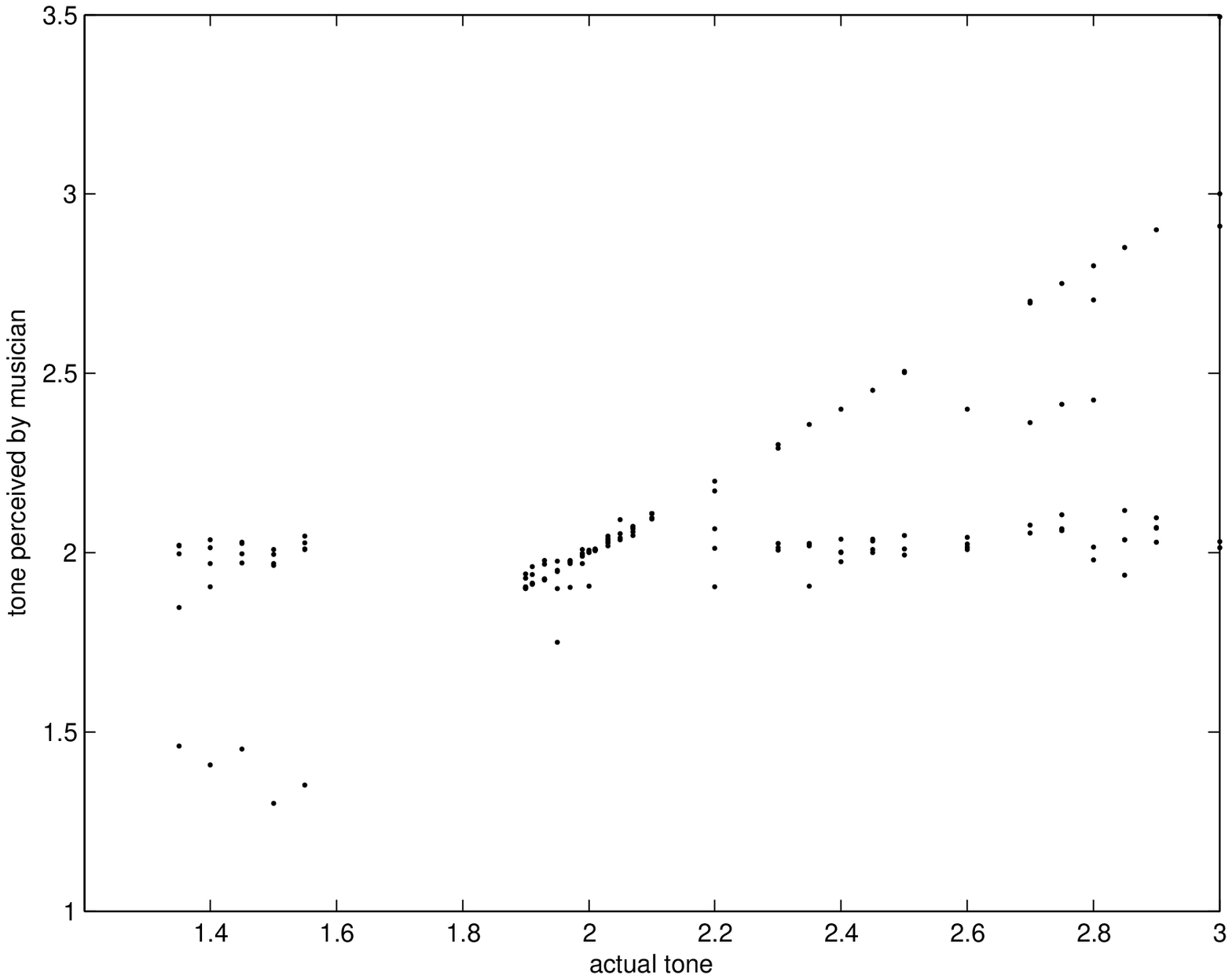}
      \caption{}
      \label{t_datapoints}
    \end{subfigure}&
    \begin{subfigure}[c]{0.45\textwidth}
      \includegraphics[scale = .4]{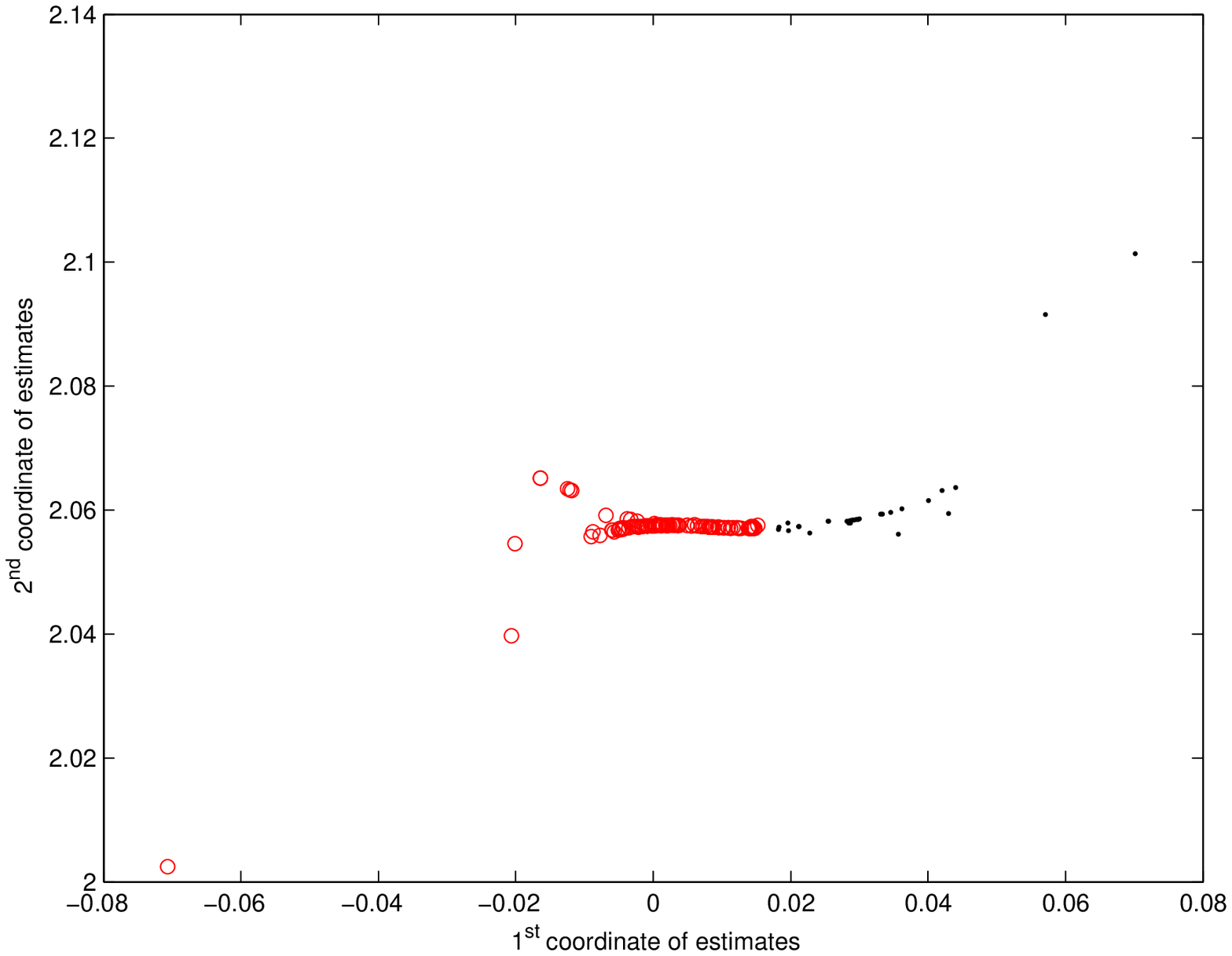}
      \caption{}
      \label{t_estimates}
    \end{subfigure}\\
    \begin{subfigure}[c]{0.45\textwidth}
      \includegraphics[scale = .4]{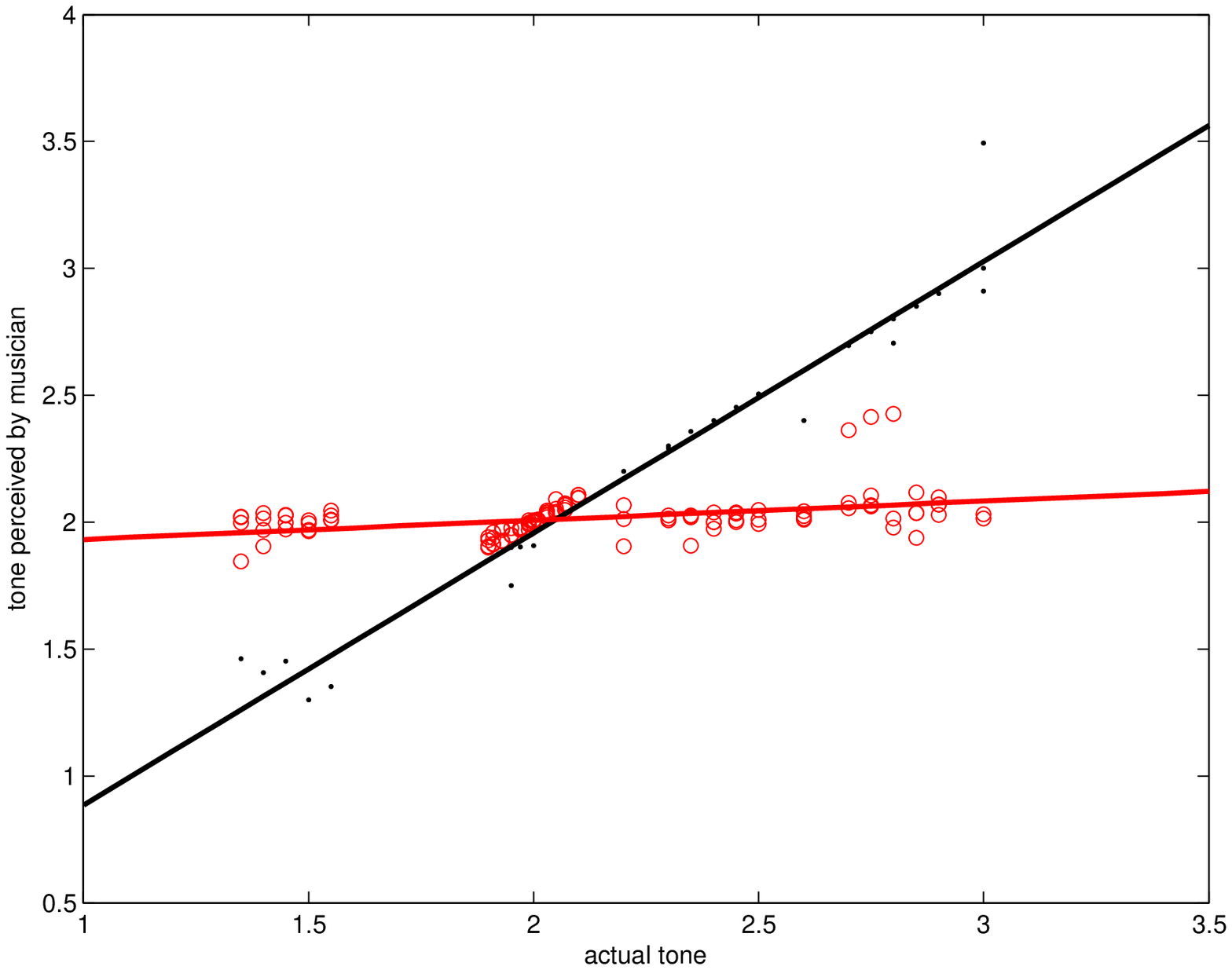}
      \caption{}
		\label{t_fit}
    \end{subfigure}&
    \begin{subfigure}[c]{0.45\textwidth}
      \includegraphics[scale = .4]{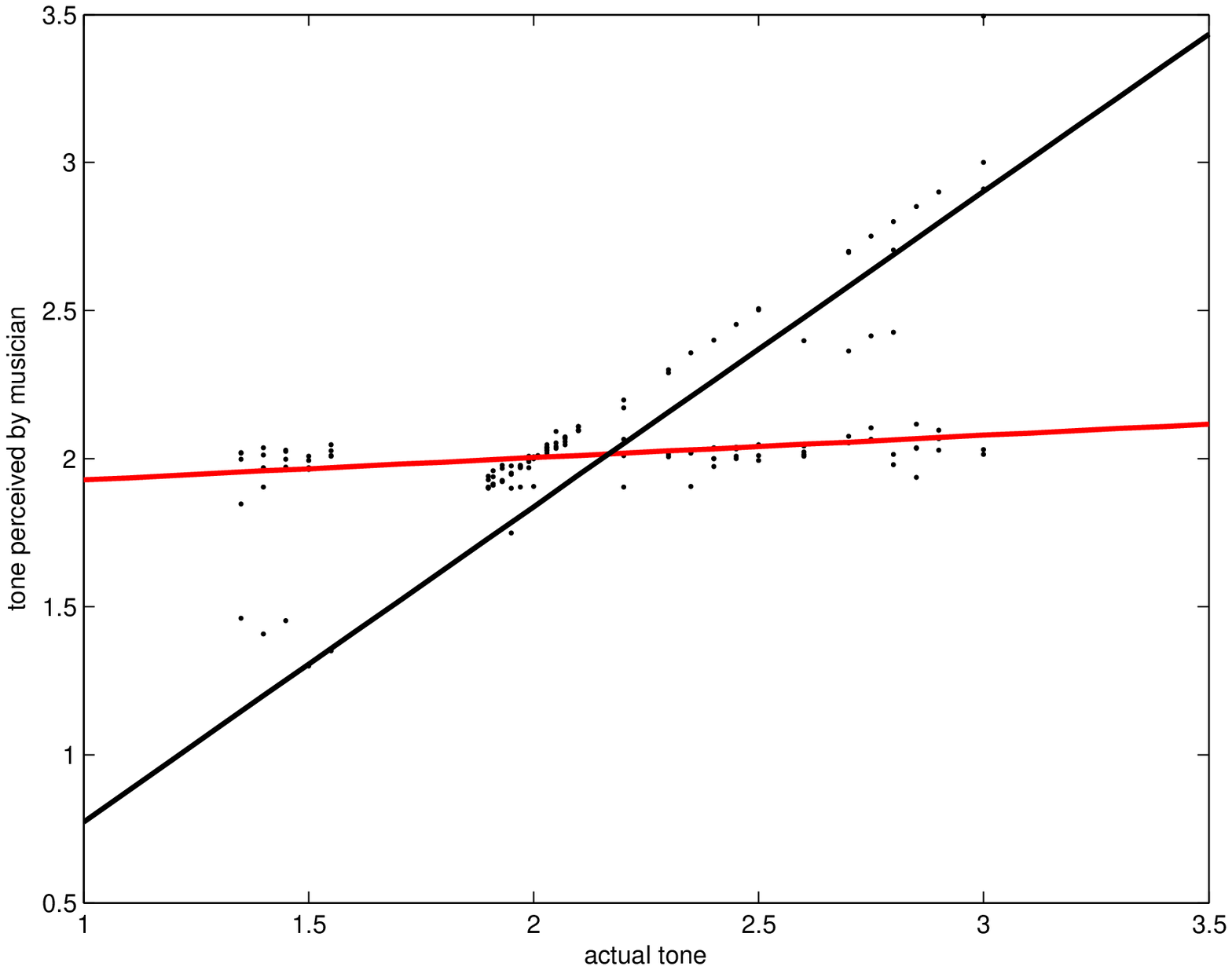}
      \caption{}
		\label{t_fit_lift}
    \end{subfigure}\\
  \end{tabular} 
  \captionsetup{oneside,margin={0em,0em}}
  \caption{Panel (a) shows perceived tone by a musician versus actual tone for a range of tones. {\change A circle or a dot in panel (b) is an estimate of one of the mixture components corresponding to a data point in panel (a). Panel (b) shows the output of \eqref{lasso} followed by $k$-means with $k = 2$. Panel (c) shows the fitted lines obtained using two separate regressions. In panel (c), points in first and second classes are represented by dots and circles, respectively. Panel (d) shows the result of Algorithm \ref{alg:Rank} with $\eta = 4$.}}
\end{figure}
\vspace{-.1cm}

First, we analyze the BikeShare data that contains the number of bikes rented in a day in the San Francisco bay area from September 1, 2014 to August 30, 2015. The data is provided by the Bay Area BikeShare program\footnote{Dataset can be obtained from the following URL: http://www.bayareabikeshare.com/open-data}. In this dataset, $m = 364$ and $d = 2$. {\change For each measurement $(a_{i},b_{i}) \in \mathbb{R}^{2}\times \mathbb{R}$, $a_{i1}=i$ is the independent variable of a simple linear model, $a_{i2} = 1$ corresponds to the constant part and $b_{i}$ is the response. In this model, $b_{i}$ is the number of bikes rented.} Let $\mu = {\change \frac{1}{m}}\sum_{i=1}^{m}a_{i1}$ and $\alpha = 0.005$. {\change Here, the average $\mu$ is used to center the dataset and $\alpha$ is used to ensure the dataset is sufficiently well-separated.} Let $\tilde{a}_{i1} = \alpha(a_{i1}-\mu)$ and let $\tilde{a}_{i} = (\tilde{a}_{i1},a_{i2})$. Algorithm \ref{alg:IRLS} is then used on the dataset $\{(\tilde{a}_{i},b_{i})\}_{i=1}^{m}$. {\change For purpose of illustration,} $k$-means with $k = 2$ is used on the output of Algorithm \ref{alg:IRLS} to estimate the mixture components. The result of this process is shown in Figure \ref{b_fit}, {\change where the classification seems to differentiate between weekend and weekday bike rental trends. For comparison, the result of Algorithm \ref{alg:Rank} on the dataset $\{(\tilde{a}_{i},b_{i})\}_{i=1}^{m}$ with $\eta$ set to 100 obtained using SDPT3 solver is shown in Figure \ref{b_fit_lift}. }

Second, we consider music tone perception data which shows the relationship between actual tone and tone perceived by a musician. This data was generated in an experiment conducted by Cohen in 1980 \cite{Cohen}. In this dataset, $m = 150$ and $d = 2$. {\change For each measurement $(a_{i},b_{i}) \in \mathbb{R}^{2}\times \mathbb{R}$, $a_{i1}=i$ is the independent variable of a simple linear model, $a_{i2} = 1$ corresponds to the constant part and $b_{i}$ is the response. In this model, $a_{i1}$ is the actual tone and $b_{i}$ is the perceived tone.} Let $\mu = {\change \frac{1}{m}}\sum_{i=1}^{m}a_{i1}$ and $\alpha = 40$. {\change Here, the average $\mu$ is used to center the dataset and $\alpha$ is used to ensure the dataset is sufficiently well-separated.} Let $\tilde{a}_{i1} = \alpha(a_{i1}-\mu)$ and let $\tilde{a}_{i} = (\tilde{a}_{i1},a_{i2})$. Algorithm \ref{alg:IRLS} is then used on the dataset $\{(\tilde{a}_{i},b_{i})\}_{i=1}^{m}$. {\change For purpose of illustration,} $k$-means with $k = 2$ is used on the output of Algorithm \ref{alg:IRLS} to estimate the mixture components. The result of this process is shown in Figure \ref{t_fit}. {\change For comparision, the result of Algorithm \ref{alg:Rank} on the dataset $\{(\tilde{a}_{i},b_{i})\}_{i=1}^{m}$ with $\eta$ set to 4 obtained using SDPT3 solver is shown in Figure \ref{b_fit_lift}.}

{\change  We also provide two simulation results on synthetic data. The first simulation result verifies Theorem \ref{thm:recovery} and the second simulation result shows recovery using $\eqref{lasso}$ in the case with imbalanced measurements is possible.}

	\begin{figure}[H]
		\centering
		\includegraphics[scale  = 0.4]{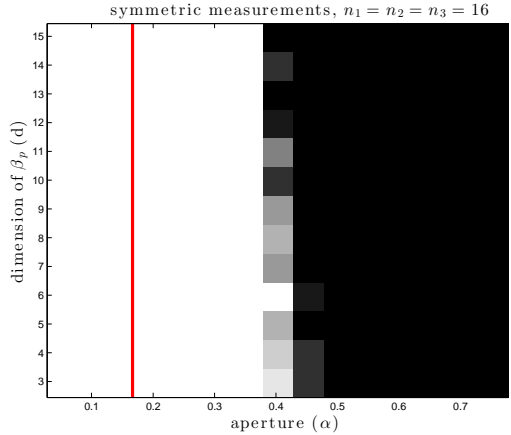}
		\captionsetup{oneside,margin={0em,0em}}
		\caption{The empirical recovery probability from synthetic data with three mixture components as a function of dimension, $d$, and aperture, $\alpha$. The shades of black and white represents the fraction of successful simulation. White blocks correspond to successful recovery and black blocks correspond to unsuccessful recovery. Each block corresponds to the average from 10 independent trials.}
		\label{fig:dimension1}
	\end{figure}

{\change For the first simulation,} let $k = 3$, $d \in \{3,\dots,15\}$ and $m = 48$. Let $\{\beta_{p}\}_{p=1}^{k}\subset \mathbb{R}^{d}$ be such that $(\beta_{p})_{l} = 1$ if $l = p$ and $0$ otherwise. Let $n_{1} = n_{2} = n_{3} = 16$. For $P_{v_{p}^{\perp}}$, let the columns of $Q_{p}\in \mathbb{R}^{d\times(d-1)}$ be  an orthonormal basis of the column space of $P_{v_{p}^{\perp}}$. 

Consider the following measurements {\change for the first simulation}: Fix $p \in {\change \{1,\dots,k\}}$. For $i \in {\change \{1,\dots,\frac{n_{p}}{2}\}}$, let $x_{i} \sim \text{Uniform}(B_{\alpha}^{d-1})$, i.e. $x_{i}$ is a random point uniformly distributed in the $d-1$ dimensional $\ell_{2}$ ball of radius $\alpha$ centered at the origin. Here, hyperplanes corresponding to labels in $S_{p}$ are contained in an aperture  $\alpha \in [0,0.75]$. For $i \in {\change \{1,\dots,\frac{n_{p}}{2}\}}$, let $a_{i} = \hat{v}_{p}+Q_{p}x_{i}$ and for $i \in \{\frac{n_{p}}{2} +1, \dots, n_{p}\}$, let $a_{i}= \hat{v}_{p}-Q_{p}x_{i-\frac{n_{p}}{2}}$. These measurements are symmetric and satisfy the balance condition since there exists pairs $i,j \in S_{p}$ such that 
\begin{equation}\label{symm}
	\frac{P_{v_{p}^{\perp}}a_{i}}{\|P_{v_{p}}a_{i}\|_{2}} = -\frac{P_{v_{p}^{\perp}}a_{j}}{\|P_{v_{p}}a_{j}\|_{2}}.
\end{equation}
Lastly, for $i \in {\change \{1,\dots,n_{p}\}}$, let $b_{i} = a_{i}^{\top}\beta_{p}$.

Figure \ref{fig:dimension1} shows the fraction of successful recovery from 10 independent trials for mixed linear regression from data as described above.  Black squares correspond to no successful recovery and white squares to 100\% successful recovery. Let $Z^{\sharp}$ be the candidate minimizer of \eqref{lasso} and let $Z^{\natural}$ be the output of $\eqref{lasso}$. For each trial, we say \eqref{lasso} successfully recovers the mixture components if $\frac{1}{\sqrt{m}}\|Z^{\natural} - Z^{\sharp}\|_{F} < 10^{-5}$. {\change This evaluation metric is used because we want to provide numerical verification of Theorem \ref{thm:recovery}.} In the area to the left of the line, the measurements are well-separated in the sense of \eqref{well-separation}, i.e. $\max_{p\in [3]}\max_{i\in S_{p}}\frac{\|P_{v_{p}^{\perp}}a_{i}\|_{2}}{\|P_{v_{p}}a_{i}\|_{2}} < \frac{1}{6}$. The figure also shows that when measurements are not well-separated, recovery using \eqref{lasso} will likely fail.     
	
\begin{figure}[H]
	\centering
	\includegraphics[scale = .4,center]{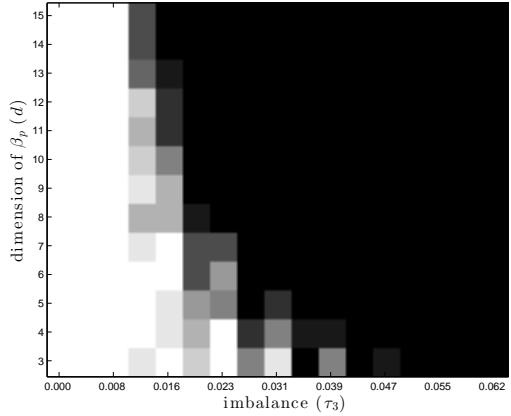} 
	\captionsetup{oneside,margin={0em,0em}}
	\caption{\change The empirical recovery probability from synthetic data with three mixture components as a function of dimension, $d$, and imbalance, $\tau_3$, of measurements in $3$rd class. The measurements in classes $1$ and $2$ are exactly balanced, i.e. $\tau_1 = \tau_2 = 0$. The shades of black and white represents the fraction of successful simulation. White blocks correspond to successful recovery and black blocks correspond to unsuccessful recovery. Each block corresponds to the average from 10 independent trials.}
	\label{fig:dimension3_avg}
\end{figure}
	
{\change For the second simulation, let $k = 3$ and $d \in \{3,\dots,15\}$. Let $\{\beta_{p}\}_{p=1}^{k}\subset \mathbb{R}^{d}$ be such that $(\beta_{p})_{l} = 1$ if $l = p$ and $0$ otherwise. Let $n_{1} = n_{2} = n_{3} = 4d$. For $P_{v_{p}^{\perp}}$, let the columns of $Q_{p}\in \mathbb{R}^{d\times(d-1)}$ be  an orthonormal basis of the column space of $P_{v_{p}^{\perp}}$, as in the first simulation.

Consider the following measurements for the second simulation: Fix the aperture, $\alpha = 0.2$. For $p \in \{1,2\}$ and $i \in \{1,\dots,\frac{n_{p}}{2}\}$, let $x_{i} \sim \text{Uniform}(B_{\alpha}^{d-1})$. For $i \in \{1,\dots,\frac{n_{p}}{2}\}$, let $a_{i} = \hat{v}_{p}+Q_{p}x_{i}$ and for $i \in \{\frac{n_{p}}{2} +1, \dots, n_{p}\}$, let $a_{i}= \hat{v}_{p}-Q_{p}x_{i-\frac{n_{p}}{2}}$. These measurements are symmetric as in the first simulation. For measurements that belong to the $3$rd class, we perturb the symmetric  measurements to achieve the desired level of imbalance. Specifically, for $p = 3$ and $i \in \{1,\dots,\frac{n_{p}}{2}\}$, let $x_{i} \sim \text{Uniform}(B_{\alpha}^{d-1})$. For $i \in \{1,\dots,\frac{n_{p}}{2}\}$, let $\tilde{a}_{i} = \hat{v}_{p}+Q_{p}x_{i}$ and for $i \in \{\frac{n_{p}}{2} +1, \dots, n_{p}\}$, let $\tilde{a}_{i}= \hat{v}_{p}-Q_{p}x_{i-\frac{n_{p}}{2}}$. Let $w \sim \text{Uniform}(S_{\tau_p}^{d-1})$ and for $i\in \{1,\dots,n_p\}$, let $a_{i} = \tilde{a_{i}}+Q_{p}w$. Here, $\tau_p \in [0,0.062]$ is the measure of imbalance of the measurements that belongs to the $p$th class. Note that 
\begin{equation}
	\tau_p = \frac{1}{n_p}\left \|\sum_{i\in S_{p}}\sign(v_{p}^{\top}a_{i})\frac{P_{v_{p}^{\perp}}a_{i}}{\|P_{v_{p}}a_{i}\|_{2}}\right\|_{2}.
\end{equation}
Lastly, for $i \in \{1,\dots,n_{p}\}$, let $b_{i} = a_{i}^{\top}\beta_{p}$.

Figure \ref{fig:dimension3_avg} show the fraction of successful recovery from 10 independent trials for mixed linear regression from data as described above. In figure \ref{fig:dimension3_avg}, the number of measurements in each class is four times the dimension. Black squares correspond to no successful recovery and white squares to 100\% successful recovery. Let $Z^{\sharp}$ be the candidate minimizer of \eqref{lasso} and let $Z^{\natural}$ be the output of $\eqref{lasso}$. For each trial, we say \eqref{lasso} successfully recovers the mixture components if $\frac{1}{\sqrt{m}}\|Z^{\natural} - Z^{\sharp}\|_{F} < 10^{-5}$. The figure shows that recovery using \eqref{lasso} from imbalanced measurements is possible if the number of measurements scale linearly with dimension of the mixture components.
}	

		
\section*{Acknowledgements}
\noindent PH acknowledges funding by the grant NSF DMS-1464525.

\bibliographystyle{abbrv}
\bibliography{MLR_191216}
\end{document}